\documentclass[reqno]{amsart}
\usepackage{amssymb}
\usepackage{a4}

\usepackage[bookmarks,hidelinks]{hyperref}
\usepackage{tikz}

\newtheorem{theorem}{Theorem}[section]
\newtheorem{lemma}[theorem]{Lemma}
\newtheorem{proposition}[theorem]{Proposition}
\newtheorem{corollary}[theorem]{Corollary}

\theoremstyle{definition}
\newtheorem{definition}[theorem]{Definition}

\numberwithin{equation}{section}

\newcommand{\N}{\mathbb{N}}

\let\midbar\mid

\let\al\mathbf \let\ab\al
\let\var\mathcal
\let\tup\mathbf

\let\equals\approx
\DeclareMathOperator{\Clo}{Clo}
\DeclareMathOperator{\Cg}{Cg}

\DeclareMathOperator{\HSP}{\mathsf {HSP}}

\newcommand{\A}{{\al A}}

\DeclareMathOperator{\Len}{l}
\newcommand{\len}[1]{\lenop({#1})}
\newcommand{\Le}[2]{\Len_{#1}({#2})}
\DeclareMathOperator{\Ht}{h}
\newcommand{\he}[1]{\htop({#1})}
\newcommand{\He}[2]{\Ht_{#1}({#2})}
\DeclareMathOperator{\fs}{Spec}
\newcommand{\Fs}[2]{\fs_{#1}({#2})}
\newcommand{\F}[2]{\al F_{#1}({#2})}
\DeclareMathOperator{\lenop}{len}
\DeclareMathOperator{\htop}{ht}

\title[Complexity of terms]%
  {Complexity of term representations of finitary functions}

\author{Erhard Aichinger}
\address{Institut f\"ur Algebra, Johannes Kepler Universit\"at Linz,
  4040 Linz, Austria}
\email{erhard@algebra.uni-linz.ac.at}

\author{Neboj\v{s}a Mudrinski}
\address{Department of Mathematics and Informatics, Faculty of Sciences,
  University of Novi Sad,
  Trg Dositeja Obradovi\'{c}a 4, 21000 Novi Sad, Serbia}
\email{nmudrinski@dmi.uns.ac.rs}

\author{Jakub Opr\v{s}al}
\address{Institut f\"ur Algebra, Technische Universit\"at Dresden,
  01062 Dresden, Germany}
\email{jakub.oprsal@tu-dresden.de}

\thanks{Supported by the Austrian Science Fund (FWF):P29931,
  Research Grant 174018 of the Ministry of Science and Education of the Republic
  of Serbia, and
  European Research Council (Grant Agreement no. 681988, CSP-Infinity).
}

\date{\today}

\begin{document}

\begin{abstract}
  The clone of term operations of an algebraic structure consists of
  all operations that can be expressed by a term in the language of the structure.
  We consider bounds for the \emph{length} and the \emph{height} of the terms
  expressing these functions, and we show that these bounds are often robust
  against the change of the basic operations of the structure.
\end{abstract}

\maketitle

\section{Motivation}

Following \cite{BS:ACIU}, an \emph{algebraic structure}, or an algebra, is
a~pair $\al A = (A,F)$ where $A$ is a~non-empty set called the \emph{universe}
of $\al A$, and $F$ is a~set of operations on $A$. The operations in $F$ are
called the \emph{basic operations} of $\al A$.  A~\emph{language} (sometimes
called \emph{signature}, or \emph{type}) of an algebra $\al A$ is a set of
symbols, one for each basic operation, together with their arities.  An $n$-ary
\emph{term} over the variables $x_1,\dots, x_n$ is any properly formed
expression using these operation symbols and variables.  A natural measure of
the complexity of a term is its length as a string, i.e., the total number of
both operation and variable symbols it contains.
We will denote this quantity by $\len t$ and call it the \emph{length of the
term $t$}. Formally, $\len{x} = 1$ if $x$ is a variable or a nullary operation
symbol, and $\len {f(t_1,\ldots, t_k)} = 1 + \sum_{i=1}^k \len {t_i}$ if $f$ is
a $k$-ary operation symbol and $t_1,\ldots,t_k$ are terms.  A term $t$ over
$x_1, \dots, x_n$  naturally induces a~function $t^{\al A}$ from $A^n$ to $A$,
which maps a~tuple $(a_1,\dots, a_n)$ to the value obtained by substituting
$a_i$ for each $x_i$ and interpreting the operation symbols by the corresponding
operations.
Next, we measure the complexity of a~term function of a~finite algebra
$\al A$: the~\emph{length of a~term function $t^{\al A}$} is the length of the
shortest term that induces this function; formally, 
\[
  \lenop_{\al A} (t^{\al A}) = \min \{ \len s :
    s \text{ is a term in the language of $\al A$ with } s^{\al A} = t^{\al A}
  \}.
\]
From this we define a~sequence $\Len_{\al A}(n)$ whose $n$-th element expresses
the worst-case complexity of an~$n$-ary term function of $\al A$:
\[
  \Len_{\al A}(n) = \max \{ \lenop_{\al A} ( t^{\al A} ) :
    \text{$t$ is an $n$-ary term in the language of $\al A$} \}.
\]
In other words, $\Len_{\al A}(n)$ is the smallest  number of symbols that is
sufficient to write down any $n$-ary term operation of $\al A$.
In the present paper we study how the asymptotics of the sequence $\Len_{\al
A}(n)$ depends on the properties of the algebra $\al A$.
The algebras $\al A$ and $\al B$ are \emph{term equivalent} if they
are defined on the same universe and have the same term operations.
Unlike other similar sequences studied in the literature, e.g., the free
spectrum, the sequences $\Len_{\al A}$ and $\Len_{\al B}$ may differ even when
$\al A$ and $\al B$ are term equivalent.  Nevertheless, in some cases we are
able to prove that the asymptotics of these sequences does not change. This
motivates the following definition.

\begin{definition}
  Let $\al A$ be a~finite algebra of finite type. Then $\al A$ is
  \emph{resilient against change of signature}, or simply \emph{resilient} if
  for every algebra $\al B$ of finite type that is term equivalent to $\al A$,
  there exist polynomials $p$, $q$ such that
  \(
    \Le{\al A}{n} \leq p(\Le{\al B}{n})
  \) and \(
    \Le{\al B}{n} \leq q(\Le{\al A}{n})
  \)
  for each $n \in \N$.
\end{definition}

It is not known to the authors whether there is a~finite algebra of finite type
which is not resilient against change of signature.  We show that two classes of
algebras are resilient. In both cases, the sequences $\Le{\al A}{n}$
are close to the theoretical lower bound given by a simple counting argument
described in the following section. These two cases are \emph{primal algebras}
and \emph{supernilpotent algebras} in congruence modular varieties. We recall
that an algebra $\al A$ with a~universe $A$ is called \emph{primal} if every
operation on $A$ is a~term operation of $\al A$.

\begin{theorem} \label{thm:primal}
  Let $\al A$ be a finite primal algebra of finite type with at least $2$
  elements.  Then there exist positive real numbers $c_1,c_2$  such that for all
  $n \in \mathbb N$,
  \(
    2^{c_1n} \leq \Le{\al A}{n} \leq 2^{c_2n}.
  \)
\end{theorem}

\begin{corollary}\label{cor:primalresilient}
  Every finite primal algebra of finite type is resilient.
\end{corollary}
The proof of these results is given in Section~\ref{sec:primal}.

The other property, supernilpotency, can be defined by a~condition on higher
commutators of the algebra. Higher commutators, introduced in \cite{Bu:OTNO},
generalize binary commutators, whose theory has been described in
\cite{FM:CTFC}. They have been studied in \cite{AM:SAOH} and recently in
\cite{Op:ARDO,Mo:HCTF,Wi:OSA}.  By a~result of Kearnes \cite{Ke:CMVW}, the
supernilpotent algebras distinguish themselves among other algebras in
a~congruence modular variety by having a~small number of term operations.
For a precise statement of this fact and a self-contained definition of
supernilpotency, we refer the reader to Section \ref{sec:spn}.

\begin{theorem} \label{thm:supernilpotent}
  Let $\al A$ be a finite supernilpotent algebra in a~congruence modular variety
  with at least $2$ elements. Then there exists a~polynomial $q$ such that for
  all $n \in \N$,
  \(
    n - 1 \leq \Le{\al A}{n} \leq q(n).
  \)
\end{theorem}

\begin{corollary}
  Every finite supernilpotent algebra of finite type in a~congruence modular
  variety is resilient.
\end{corollary}

These results are proved in Section~\ref{sec:spn}.  For groups, Horv\'{a}th and
Nehaniv have proved that for each nilpotent group $\ab{G}$, $\Le{\ab{G}}{n}$ is
bounded from above by a polynomial in $n$, and for simple groups, one has
\(
  \Le{\ab{G}}{n} \le C (\ab{G}) \cdot  n^8 \cdot |G|^n
\)
where $C(\ab{G})$ depends only on the group $\ab{G}$ \cite{HN:LOPO}.

\section{Notation and general bounds}

We will write $\N$ for the set of positive integers, $\N_0 = \N \cup \{0\}$, and
$\log (n)$ is the logarithm of $n$ in base $e$.  For an~algebra $\al A$, the set
of all term operations is denoted by $\Clo \al A$ and called the clone of term
operations of $\al A$. Its $n$-ary part is denoted by $\Clo_n \al A$.  $\Clo_n
\al A$ can be endowed with a~structure of an algebra, denoted by $\F{\al A}{n}$,
in the same language as $\al A$ by setting:
\[
  f^{\F{\al A}{n}} (t_1^{\al A},\dots,t_k^{\al A}) =
  (f (t_1,\dots,t_k))^{\al A},
\]
or equivalently defining the operations coordinatewise, seeing $\Clo_n\al
A$ as a~subuniverse of $\al A^{A^n}$.
It is well-known that this algebra is freely generated (in $\HSP \al A$) by the
projections (the set $\{ x_1^{\al A},\dots,x_n^{\al A}\}$). The sequence defined
by
\[
  \Fs{\al A}{n} = \left| \F{\al A}{n} \right|
\]
for $n \in \N$ is called the \emph{free spectrum} of $\al A$.

We will now introduce another measure for the complexity of a~term. Every term
can be expressed as a~rooted tree whose vertices correspond to the function and
variable symbols appearing in it. The variables correspond to leaves, each
function symbol has exactly as many children as is the arity of the symbol, and
the children have a~(usually implicitly) given order. We call this tree the
\emph{tree representation} of a~term $t$.
The number of vertices of the tree representing $t$ is exactly $\len t$.
The \emph{height} of the term $t$, denoted by $\he t$, represents the height of
the tree representation of $t$, with an adjustment if $t$ contains a~nullary
operation symbol:
Precisely, we define $\he{t}$ inductively by setting:
$\he{x} = 0$ if $x$ is a variable;
$\he{f} = 1$ if $f$ is a~nullary operation symbol; and
if $f$ is a~$k$-ary operation symbol where $k\in \mathbb N$, and $t_1,\dots,t_k$ are
terms, we set $\he{f(t_1,\dots,t_k)} = 1 + \max\{\he{t_1},\dots,\he{t_k}\}$.
Through the rest of the present paper, we set $\max \emptyset = 0$. This will
simplify some proofs since it saves a case distinction between nullary and
non-nullary operation symbols.

The \emph{height of a~term function} $t^\A$ is defined by
\[
\htop_{\al A} (t^{\al A}) = \min \{ \he s :
    s \text{ is a term in the language of $\al A$ with } s^{\al A} = t^{\al A} \}.
\]
This allows us to define a~second sequence measuring the complexity of
terms of an~algebra $\al A$. For $n \in \N$, let
\[
  \Ht_{\al A}(n) = \max \{ \htop_{\al A} ( t^{\al A} ) :
    \text{$t$ is an $n$-ary term in the language of $\al A$} \}.
\]

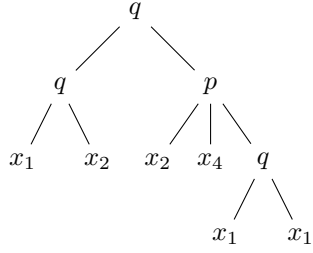
\begin{figure}
    \begin{tikzpicture}
      \node (q) at (0,0) {$q$};
        \node (q1) at (-1,-1) {$q$};
          \node (x11) at (-1.5,-2) {$x_1$};
          \node (y12) at (-.5,-2) {$x_2$};
        \node (p2) at (1,-1) {$p$};
          \node (y21) at (.3,-2) {$x_2$};
          \node (w22) at (1,-2) {$x_4$};
          \node (q23) at (1.7,-2) {$q$};
            \node (x231) at (1.2,-3) {$x_1$};
            \node (x232) at (2.2,-3) {$x_1$};
        \draw (p2) -- (q) -- (q1);
        \draw (x11) -- (q1) -- (y12);
        \draw (y21) -- (p2) -- (w22); \draw (p2) -- (q23);
        \draw (x231) -- (q23) -- (x232);
    \end{tikzpicture}
    \caption{Tree representation of term
      \(
        q( q( x_1, x_2 ), p( x_2, x_4, q( x_1, x_1 )))
      \).}
\end{figure}

\begin{lemma} \label{lem:general-bound}
  Let $\al A$ be a~finite algebra of finite type, let $r$ be the number of
  operation symbols of $\al A$ and $m$ be their maximal arity.
  Then:
  \begin{enumerate}
  \item \label{it:1}  $\Fs{\A}{n} < (n+r+1)^{\Len_\A (n)}$ for all $n \in \N$;
    \item \label{it:1a} there exists a~positive real number $c$ such that
      for all $n \ge 2$ we have $\Fs{\A}{n} \leq 2^{c \log (n) \Len_{\al A} (n)}$;
    \item \label{it:2} $\Ht_\A (n) \leq \Fs{\A}{n}$ for all $n \in \N$;
    \item \label{it:3} $\He{\al A}{n} \leq \Le{\al A}{n}$ for all $n \in \N$;
    \item \label{it:4} $\Le{\al A}{n} \leq 1 + m + \dots + m^{\Ht_{\al A}(n)}
      \leq (m+1)^{\Ht_{\al A}(n)}$ for all $n \in \N$.
  \end{enumerate}
\end{lemma}

\begin{proof}
  \eqref{it:1} By the definition of the sequence $\Len_\A(n)$, every $n$-ary term
  operation is given by a~term whose representation in prefix notation is
  a~string of length at most $\Len_\A(n)$ using $n$ symbols for variables and
  $r$ symbols for basic operations.
  Therefore, there are at most $\sum_{l=1}^{\Len_\A(n)} (n+r)^l <
  (n+r+1)^{\Len_{\al A}(n)}$ term operations of $\al A$ of arity $n$.  Now
  \eqref{it:1a} follows because there is a positive real number $c$ such that
  $\log_2 (n+r+1) \leq c \log n$ for all $n \geq 2$.

  \eqref{it:2}
  Let $F \le \al A^{A^n}$ be the universe of the free algebra generated
  by the projections $\pi_1,\dots,\pi_n$, i.e., the elements of $F$ are
  precisely the $n$-ary term functions of $\al A$.
  We consider the following alternative construction of the set $F$:
  We start by taking $H_0 = \{ \pi_1,\dots,\pi_n \}$, and then
  inductively define sets $H_k$ by
  \[
    H_{k+1} = \{ f^{\al A^{A^n}} ( t_1, \dots, t_m ) :
    t_1,\dots, t_m\in H_k,
    f \text{ is a~basic operation of $\al A$} \}\cup H_k
  \]
  for $k \in \N_0$.  If for some $k$, $H_{k+1} = H_k$, then  $H_k$ is closed
  under the operations of $\al A^{A^n}$,  therefore a~subuniverse, and hence $F
  = H_k$. Let $k_0$ be the smallest $k$ such that $H_k = H_{k+1}$.  Then $F =
  H_{k_0}$. Now, since $|H_{k+1}| \geq |H_k| + 1$ for all  $k< k_0$, we obtain
  $|F| \geq k_0$.
  Further, observe that for any $k$, $H_k$ is the set of all $n$-ary term
  operations of $\al A$ which can be expressed by a~term of depth at most $k$,
  in other words,
  \(
    H_k = \{ f^{\al A} \in \Clo_n(\al A) : \he{f^{\al A}} \leq k \}
  \).
  Therefore, since $H_{k_0} = F$, we obtain that $\he{f^{\al A}} \leq k_0 \leq
  |F|$ for any $n$-ary term operation $f^{\al A}$ of $\al A$, which immediately
  yields the required inequality.

  \eqref{it:3}
  Let $t$ be an $n$-ary term. Then there is a term $s$ with $\len{s} \le \Le{\al
  A}{n}$ such that $t^{\al A} = s^{\al A}$. For every term, we have $\len{s} \ge
  \he{s}$, and therefore $\he{s} \le \Le{\al A}{n}$.  This implies
  $\He{\al A}{n} \le \Le{\al A}{n}$.

  \eqref{it:4}
  We first observe that for every term $t$,
  \begin{equation}\label{lenht}
    \len {t} \leq 1 + m + \dots + m^{\he{t}}
  \end{equation}
  because a rooted tree in which every vertex has at most $m$ children has at
  most $m^d$ vertices of depth $d$.
  In order to obtain an upper bound for $\Le{\al A}{n}$, we let $t$ be an
  $n$-ary term. Hence $\he{t^{\al A}} \le \He{\al A}{n}$, and hence there is
  a~term $s$ with $\he{s} \leq \He{\al A}{n}$ and $s^{\al A} = t^{\al A}$. Now
  \eqref{lenht}  implies that $\len{s}$, and therefore also $\lenop_{\al A}
  (t^{\al A})$, is at most $\sum_{i = 0}^{\He{\al A}{n}} m^i$ which completes
  the proof of the first inequality.  The second inequality is immediate.
\end{proof}
 
For an $n$-ary term $t$ and terms $t_1,\ldots, t_n$, we write $t(t_1,\ldots,
t_n)$ for the term obtained by substituting every occurrence of $x_i$ in $t$ by
$t_i$.

\begin{lemma}\label{lem:heightfromthemiddle}
  Let $n\in\N$, let $\mathcal{L}$ be a language, and let $t,t_1,\dots,t_n$
  be terms of type $\mathcal{L}$. 
  If $t$ is $n$-ary, then
  \(
    \he{t(t_1,\dots,t_n)} \leq \he{t} + \max\{\he{t_1},\dots,\he{t_n}\}
  \).
\end{lemma}

\begin{proof}
  We proceed by induction on $\he{t}$: If $t$ is a variable $x_i$, then
  $\he{t} = 0$, and the statement reduces to
  $\he{t_i} \le \max\{\he{t_1},\dots,\he{t_n}\}$ which is trivial.
  For the induction step, we assume
  that $t = f(s_1,\dots,s_k)$ for
 some operation symbol $f$ and $n$-ary terms $s_1,\dots,s_k$. To simplify the
formulae below, let $m = \max\{\he{t_1},\dots,\he{t_n}\}$.
Observe that $\he{s_i}
\leq \he{t} - 1$. From the induction hypothesis we obtain
\[
  \he{s_i(t_1,\dots,t_n)} \leq \he{s_i} + m
    \leq \he{t} + m - 1
\]
for every $i \in \{1,\dots,k\}$. Therefore,
\begin{multline*}
  \he{t(t_1,\dots,t_n)}  =
  \he{f(s_1 (t_1, \dots,  t_n),\dots, s_k (t_1,\ldots,t_n))} \\
  =
  1 + \max\{\he{s_1(t_1,\dots,t_n)},\dots,
    \he{s_k(t_1,\dots,t_n)}\} \leq \he{t} + m.
\end{multline*}
\end{proof}

\begin{lemma}\label{termequivalent}
  Let $\al A$ and $\al B$ be two term equivalent finite algebras
  of finite type. Then there exist positive real numbers $c_1,c_2$ such that for
  all $n \in \N$, 
  \[
    c_1\He{\al B}{n} \leq \He{\al A}{n} \leq c_2\He{\al B}{n}.
  \]
\end{lemma}

\begin{proof}
  We first consider the second inequality.
  We choose $c_2$ such that every basic operation $f^{\al B}$ of
  $\al B$ can be expressed as a~term operation of $\al A$ of height at most $c_2$;
  such a $c_2$ exists because $\al B$ is of finite type and the two algebras
  are term equivalent.
  Now, given that $t^{\al A}$ is a~term operation of $\al A$ of arity $n$, and $\al
  A$ and $\al B$ are term equivalent, then $t^{\al A}$ is also a~term operation
  of $\al B$. Therefore, there exists a~term $s$ in the language of $\al B$ of
  height at most $\He{\al B}{n}$ such that $s^{\al B} = t^{\al A}$.
  We will turn this term into a~term $r$ in the language of $\al A$ with
  $r^{\al A} = s^{\al B}$ by substituting every basic operation symbol by its
  definition as a~term of $\al A$ (again, we use that $\al A$ and $\al B$ are
  term equivalent). The term $r$ which is obtained this way will be of height at
  most $c_2\he{s} \leq c_2\He{\al B}{n}$.
  Formally, we prove the existence of a term $r$ with $r^{\al A} = s^{\al B}$
  and $\he{r} \le c_2 \he{s}$ by induction on the height of $s$.
  If $s$ is a~variable, we set $r := s$ and obtain the required inequality
  from the fact that $\he{r} = \he{s} = 0$.
  Now assume that $s = f(s_1,\dots,s_k)$ for some $k \in \N_0$, some
  terms $s_1,\dots,s_k$ and a~basic operation $f$ of $\al B$. Then
  $\he{s_i} \leq \he{s} - 1$ for all $i$. By the induction hypothesis, we know
  that there are terms $r_1,\dots,r_k$ such that
  \[
    \he{r_i} \leq c_2 \he{s_i} \leq c_2(\he{s} - 1)
    \text { and }
    r_i^{\al A} = s_i^{\al B}.
  \]
  There is also a~term $g$ in the language of ${\al A}$ of height at most $c_2$
  such that
  $g^{\al A} = f^{\al B}$.
  Let  $r = g(r_1,\dots,r_n)$.
  Then
  \(
    r^{\al A} =
      g^{\al A}( r_1^{\al A}, \dots, r_k^{\al A} ) =
      f^{\al B}( s_1^{\al B}, \dots, s_k^{\al B} ) =
    s^{\al B}
  \).
  Moreover, from Lemma~\ref{lem:heightfromthemiddle} we obtain that
  \(
    \he{r} \le \he{g} + \max\{\he{r_i} : i \in \{ 1,\dots,k\} \}
      \leq c_2 + c_2(\he{s} - 1)
      = c_2\he{s} \leq c_2 \He{\al B}{n}
  \).
  This completes the induction step.
  Therefore $\He{\al A}{n} \le c_2 \He{\al B}{n}$.

  For proving the first inequality, we interchange the roles of ${\al A}$ and
  ${\al B}$ to obtain a $c' > 0$ such that for all $n \in \N$,
  $\He{\al B}{n} \le c' \He{\al A}{n}$. Then  the inequality
  is satisfied with
  $c_1 := 1/c'$.
\end{proof}

\begin{corollary}
  Let $\al A$ and $\al B$ be two term equivalent finite algebras
  of finite type. Then there exist positive real numbers $d_1,d_2$ such that for
  all $n \in \N$, 
  \[
    d_1 \log_2 (\Le{\al B}{n}) \le \Le{\al A}{n} \le 2^{d_2 \Le{\al B}{n}}.
  \]
\end{corollary}

\begin{proof}
  Let $c_2$ be the constant from the previous lemma, i.e., we have $\He{\al A}n
  \leq c_2\He{\al B}n$ for all $n\in \mathbb N$.  Let $m$ be the maximal arity
  of operation symbols of $\al A$, and let $d_2 := c_2 \log_2 (m + 1)$.
  By Lemma~\ref{lem:general-bound}\,\eqref{it:4}, we have
  \(
    \Le{\al A}{n} \le (m+1)^{\He{\al A}{n}}.
  \)
  Using the inequality on the height, we obtain
  \[
    (m+1)^{\He{\al A}{n}} \le (m+1)^{c_2 \He{\al B}{n}}
      = 2^{d_2 \He{\al B}{n}}.
  \]
  Now by Lemma~\ref{lem:general-bound}~\eqref{it:3}, we have $2^{d_2 \He{\al
  B}{n}} \le 2^{d_2 \Le{\al B}{n}}$, and therefore the second inequality holds.
  For proving the first inequality, we interchange the roles of ${\al A}$ and
  ${\al B}$ to obtain a $d' > 0$ such that for all $n \in \N$, $\Le{\al B}{n}
  \le 2^{d' \He{\al A}{n}}$.
  Hence $\log_2 (\Le{\al B}{n}) \le d'\He{\al A}{n}$. Hence the required
  inequality is satisfied with $d_1 := 1/d'$.
\end{proof}

We are not aware of finite term equivalent algebras $\al A$ and $\al B$
of finite type such that there is a positive real number $c$ with
$\Le{\al A}{n} \ge 2^{c \Le{\al B}{n}}$ for all $n \in \N$.
A natural example that the change of signatures can make it significantly easier
to write down certain functions is given by the commutator operation in a group.
Let  $\al A = \al A_4$ to be the alternating group on four elements, and let
$\al B$ be the algebra obtained from $\al A$ by adding a~single binary operation
$[{\cdot},{\cdot}]$ expressing the commutator of the two elements, i.e., $[x,y]
= x^{-1}y^{-1}xy$. Observe that while the $n$-ary iterated commutator $t_n =
[\dots\![[x_1,x_2],x_3],\dots x_n]$ has linear length in $n$, the corresponding
term of $\al A$ (the one obtained by simply substituting the definition of the
commutator for each of its appearance) has length more than $2^n$.
Unfortunately, we are unable to prove that the term functions $t_n^{\al B}$
cannot be represented by terms in the language of ${\al A}$  whose length would
be bounded by a polynomial in $n$.  Nevertheless, using \cite{HS:EAES} and under
the assumption that $\mathrm{P} \neq \mathrm{NP}$, such terms cannot be produced
in polynomial time:

\begin{proposition}
  Let ${\al A} := (A_4, \cdot, \mbox{}^{-1}, 1)$ be the alternating group
  on four letters, let ${\al B} := (A_4, \cdot, \mbox{}^{-1}, 1, [.,.])$
  be its expansion with the commutator operation $[x,y] := x^{-1}y^{-1} x y$,
  let $t_1 = x_1$, and for $n \in \N$, let $t_n := [t_{n-1}, x_n]$.
  If $\mathrm{P} \neq \mathrm{NP}$, then there is no algorithm which, given $n$,
  produces a~term $s_n$ in the language of $\al A$ such that ${s}_n^{\al A} =
  t_n^{\al B}$ and which runs in polynomial time in $n$.
\end{proposition}

\begin{proof}
  We derive this result from results of \cite{HS:EAES}: the equation solvability
  problem for $\al A$ is in $\mathrm{P}$ \cite[Theorem 6]{HS:EAES}.
  Furthermore, we will use their reduction of $3$-colorability to the equation
  solvability problem for $\al B$ \cite[Theorem 13]{HS:EAES}.

  Suppose that there is a~polynomial time algorithm producing $s_n$.  We will
  use this algorithm to reduce $3$-colorability to equation solvability in $\al
  A$. To this end, let $\Gamma = (V,E)$ be a~graph with $n$ vertices and $k$
  edges, and let $V = \{v_1,\dots,v_n\}$ and $E = \{e_1,\dots,e_k\}$.  In
  \cite{HS:EAES}, the authors produce a term $t_\Gamma$ in the language of $\al
  B$ of polynomial length in $n$ and $a \in A_4$ such that $t_\Gamma \equals a$
  has a solution for $a\neq 1$ if and only if the graph $\Gamma$ is
  $3$-colorable. This term is defined by
  \[
    t_{\Gamma}( x_1, x_2, y_1, \dots, y_n ) =
    [\dots\![[x_1, x_2], g_{e_1}], \dots \dots, g_{e_k}],
  \]
  where $g_{(v_i,v_j)} = y_i^{\vphantom{-1}}y_j^{-1}$.
  The term $t_{\Gamma}$ is of a~polynomial length in $n$ as a~term of $\al B$.
  Using this term, Horv\'{a}th and Szab\'{o} reduce
  $3$-colorability to equation solvability in $\al B$.
  We have to do one more step to reduce it to equation solvability in $\al A$, that
  is find a~polynomial length term in the language of groups. For that we run
  the given algorithm to produce a~term $s_{k+2}$ in the language of $\al A$
  such that $s_{k+2}^{\al A} =t_{k+2}^{\al B}$. Note that this term has to
  be of polynomial length in $k$ since it is produced by a~polynomial time
  algorithm.
  Finally, let
  \[
    s_{\Gamma}( x_1, x_2, y_1, \dots, y_n ) =
    s_{k+2}(x_1, x_2, g_{e_1}, \dots \dots, g_{e_k}).
  \]
  This term can be computed from $s_{k+2}$ in linear time, and since
  $s_\Gamma^{\al A} = t_\Gamma^{\al B}$, the equation $t_\Gamma
  \equals a$ has a solution in $A_4$ if and only if $s_\Gamma \equals a$ does.
  Hence $\Gamma$ is $3$-colorable if and only if
  $s_{\Gamma} \equals a$ is solvable. 
  This completes
  the reduction from 3-colorability to equation solvability in $\al A$.
  By \cite[Theorem 6]{HS:EAES}, equational solvability in $\al A$ is in $\mathrm{P}$.
  Altogether, we have produced a polynomial time algorithm for
  $3$-colorability. Since $3$-colorability is $\mathrm{NP}$-complete  \cite{Ka:RACP},
  this contradicts the assumption $\mathrm{P} \neq \mathrm{NP}$.
  Therefore, assuming $\mathrm{P} \neq \mathrm{NP}$, such an algorithm
  producing $s_n$ does not exist.
 \end{proof}

We will also use the following elementary lemma.
\begin{lemma}\label{lem:recursion}
  Let $f\colon \mathbb N \to \mathbb R$ be a~non-decreasing function, let $n_0 \in
  \mathbb N$, $0 < c < 1$, and $d > 0$ such that $c n_0 \ge 1$ and
  for all $n \geq n_0$ we
  have
  \(
    f( n ) \leq f( \lfloor cn \rfloor ) + d.
  \)
  Then for all $n \in \N$, we have
  \(
    f( n ) \leq d\log_{1/c} n + f(n_0).
  \)
\end{lemma}

\begin{proof}
  We will prove the claim by induction on $n$. For $n \leq n_0$, the statement
  follows from monotonicity of $f$. For the induction step, suppose that $n >
  n_0$ and the statement is true for all $m < n$. We obtain
  \begin{multline*}
    f( n ) \leq f( \lfloor cn \rfloor ) + d
    \leq d\log_{1/c} (\lfloor cn \rfloor ) + f(n_0) + d \\
    \leq d( 1 + \log_{1/c} ( cn ) ) + f(n_0)
    = d\log_{1/c} n + f(n_0)
  \end{multline*}
  from the induction hypothesis and the monotonicity of the logarithm.
\end{proof}

\section{Primal algebras} \label{sec:primal}

In this section, we will consider sums $x_1 + x_2 + \dots + x_n$, where $+$ is
some binary operation.  These sums can be parenthesized in various ways. We will
put the parentheses in a way that yield a balanced binary tree.

\begin{definition}
  Let $t$ be a binary symbol. We define the sequence $(\sigma_n^t)_{n\in\N}$ of
  terms in the language $\{t\}$ by $\sigma_1^t(x_1):=x_1$ and
  \[
    \sigma_n^t(x_1,\dots,x_n) := 
    t( \sigma_{\lceil n/2 \rceil}^t(x_1,\dots,x_{\lceil n/2 \rceil}),
      \sigma_{\lfloor n/2 \rfloor}^t(x_{\lceil n/2 \rceil +1},\dots,x_n) )
  \]
  for $n \ge 2$.
\end{definition}

\begin{lemma}\label{allbutone}
 Let $t$ be a binary symbol in the language of an algebra $\al A$ and $0\in A$
 such that $t^{\al A}(x,0) = t^{\al A}(0,x) = x$ for all $x\in A$. Then for all
 $n\in\N$ and $x \in A$, we have $({\sigma_n^t})^{\al A}(0,\dots,0,x,0,\dots,0)
 = x$.
\end{lemma}

\begin{proof}
  By induction on $n$. For $n=1$ the statement is obviously true by the
  definition of $(\sigma_n^t)_{n\in\N}$. Let $n\geq2$, and let $x \in A$.  We
  first consider the case that $x$ appears at place $i$ with $i > \lceil n/2
  \rceil$. Then the equalities ${\sigma_{\lfloor n/2\rfloor}^t}^{\al
  A}(0,\dots,0)=0$ and ${\sigma_{\lceil n/2\rceil}^t}^{\al
  A}(0,\dots,0,x,0,\dots,0) = x$ follow from the induction hypothesis. Using the
  assumption $t^{\al A}(0,x)=x$, we obtain $({\sigma_n^t})^{\al
  A}(0,\dots,0,x,0,\dots,0)=x$.  The case $i \le \lceil n/2 \rceil$ is done
  similarly.
\end{proof}

The height of a~binary balanced tree can be determined from its number of leaves.
In our setting, this means:
\begin{lemma}\label{heightofsigma} Let $t$ be a~binary symbol. Then
  $\he{\sigma_n^t} = \lceil \log_2 n \rceil$.
\end{lemma}

\begin{proof}
  We prove the statement
  \[
  \forall k \in \N_0 \, \forall n \in \N \,:\,
    2^{k-1} < n \le 2^k \Rightarrow
    \he{\sigma_n^t} = k
    \]
   by induction on $k$.
  It is clearly true for $k \in \{0,1\}$.
  Now assume  that $k > 1$.    Since $2^{k-1} < n \leq 2^k$,
  have $2^{k-2} < n/2 \leq 2^{k-1}$, and therefore,
  $2^{k-2} < \lceil n/2 \rceil \leq 2^{k-1}$ and $\lfloor n/2 \rfloor \leq
  2^{k-1}$. From these inequalities and the induction hypothesis, we obtain that
  $\he{ \sigma^t_{ \lceil n/2 \rceil } } = k - 1$ and
  $\he{ \sigma^t_{ \lfloor n/2 \rfloor } } \leq k - 1$. Hence, from the definition
  of $\sigma_n^t$, we get $\he{ \sigma_n^t } = k$. This completes the induction proof.
  Now we notice that if $2^{k-1} < n \le 2^k$, then $k = \lceil \log_2 n \rceil$,
  which implies the result. 
\end{proof}

\begin{lemma}\label{lem:specialfunctionallycomplete}
  Let $\al A$ be a finite algebra that has binary
  operations $+$ and $\cdot$ as fundamental operations as well as
  characteristic functions $\chi_{a}$ for all $a\in A$ such that
  $x+0=0+x=x$, $x\cdot 1=x$ and $x\cdot 0=0$ for all $x\in A$. If
  all unary constant operations of $\A$ are term functions then there is
  a positive real number $d$ such that $\He{\al A}{n}\leq dn$ for all $n\in\N$.
\end{lemma}

\begin{proof}
  Let $n \in \N$, let $|A|=m$, $m\in\N$, let $A=\{a_1,\dots,a_m\}$ and let us
  denote $\Pi_{i=1}^n\chi_{\alpha_i}(x_i) =
  (\dots(\chi_{\alpha_1}(x_1)\cdot\chi_{\alpha_2}(x_2))\cdot\ldots) \cdot
  \chi_{\alpha_n}(x_n)$ by $\chi_{\alpha_1}(x_1)\cdots\chi_{\alpha_n}(x_n)$ for
  all $\alpha_1,\dots,\alpha_n\in A$ and $x_1,\dots,x_n\in A$. By Lemma
  \ref{allbutone} we have that arbitrary term function $f$ can be represented as
  \begin{equation}\label{eq:0}
    {\sigma_{m^n}^+}(f(a_1,\ldots,a_1)\cdot\chi_{a_1}(x_1)\cdots\chi_{a_1}(x_n),
      \ldots,f(a_m,\ldots,a_m)\cdot\chi_{a_m}(x_1)\cdots\chi_{a_m}(x_n)).
  \end{equation} 
  We note that $\he{x+y}=\he{x\cdot y}=1$ and $\he{\chi_{a_i}(x)}=1$ for all
  $i\in\{1,\dots,m\}$. Let $s$ be the maximal height of the terms that represent
  constant functions.  Now we can calculate the height of \eqref{eq:0}. Using
  Lemma \ref{heightofsigma}, we have $\he{\sigma_{m^n}^+}\leq
  \lceil\log_2|A|^n\rceil$. Using the definition of the  height we obtain
  \[
    \he{f(\alpha_1,\ldots,\alpha_n)\cdot\Pi_{i=1}^n\chi_{\alpha_i}(x_i)} =
    1 + \max\{ \he{f(\alpha_1,\ldots,\alpha_n)},
      \he{\Pi_{i=1}^n\chi_{\alpha_i}(x_i)} \}
  \]
  and $\he{\Pi_{i=1}^n\chi_{\alpha_i}(x_i)} = n$ for all
  $\alpha_1,\dots,\alpha_n\in A$. Hence, for all $n \in \N$, we have 
  \[
    \he{f(\alpha_1,\ldots,\alpha_n)\cdot\Pi_{i=1}^n\chi_{\alpha_i}(x_i)}
      \leq 1 + \max\{s,n\}.
  \]
  Hence there is $c\in\N$ such that
  $\he{f(\alpha_1,\ldots,\alpha_n)\cdot\Pi_{i=1}^n\chi_{\alpha_i}(x_i)}\leq cn$
  for all $n \in \N$.  Therefore, for each $n \in \N$,  by
  Lemma~\ref{lem:heightfromthemiddle} the height of \eqref{eq:0} is at  most
  $\lceil\log_2|A|^n\rceil+cn$. There is a positive real $d$ such that for all
  $n \in \N$, $\lceil\log_2|A|^n\rceil+ cn \le dn$. Hence $\He{\A}{n}\leq dn$.
\end{proof}

\begin{proposition}\label{thm:functionallycomplete}
  Let $\al A=(A,F)$ be a finite  primal algebra of finite type, and let
  $n\in\N$. Then there is a positive real number $c$ such that $\Le{\al A}{n}
  \leq 2^{cn}$ for all $n \in \N$.
\end{proposition}

\begin{proof}
Let $\al A$ be a primal algebra and let $u$ be the  maximal arity of its
fundamental operations. By \cite[Theorem 3.1.5]{KP:PCIA} every primal algebra
$\al A=(A,F)$ such that $A=\{a_1,\ldots,a_m\}$, is term equivalent to $\mathbf
B=(A,+,\cdot,\chi_{a_1},\ldots,\chi_{a_m},0,1)$ where
$\chi_{a_1},\ldots,\chi_{a_m}$ are characteristic functions, $0$ and $1$ are
elements from $A$ and $+$ and $\cdot$ are binary operations such that
$x+0=0+x=x$, $x\cdot 1=x$, $x\cdot 0=0$ for all $x\in A$ and every constant
function is a term function of $\mathbf B$.  By Lemma \ref{termequivalent},
there is a positive real number $a$ such that $\He{\al A}{n}\leq a\He{\al B}{n}$
for all $n \in \N$.  Using Lemma \ref{lem:specialfunctionallycomplete}, we
obtain a positive real  $d$ such that $\He{\al A}{n}\leq dn$ for all $n \in \N$.
Using Lemma \ref{lem:general-bound} we have $\Le{\A}{n}\leq(u+1)^{dn}$, and thus
we can find a positive real $c$ such that $\Le{\al A}{n}\leq 2^{cn}$ for all $n
\in \N$.
\end{proof}

\begin{proof}[Proof of Theorem \ref{thm:primal}.]
For every primal algebra $\A$ and $n\in\N$ we have $\Fs\A{n}=|A|^{|A|^n}$. Now
the first inequality follows from Lemma \ref{lem:general-bound}\,\eqref{it:1a}.
The second inequality is given in Proposition \ref{thm:functionallycomplete}.
\end{proof}

\section{Supernilpotent algebras} \label{sec:spn}

The notion of supernilpotency was introduced in \cite[Definition~4.1]{AE:EKCN}
for expansions of groups, and in its general form in
\cite[Definition~7.1]{AM:SAOH}. It is closely related to Bulatov's higher
commutators introduced in \cite{Bu:OTNO}.

\begin{definition}
An~algebra $\al A$ is said to be \emph{supernilpotent} of degree $n$ if it
satisfies the commutator identity $[1_{\al A},\dots, 1_{\al A}] = 0_{\al A}$
where the commutator has arity $n+1$.
\end{definition}

The commutator identity in the above definition is described by the
following term condition: for all $i=1,\dots,n+1$, $k_i \in \mathbb N$, $\tup
a_i,\tup b_i \in A^{k_i}$  with $\tup a_i \neq \tup b_i$, and all terms $t$ of
arity $\sum k_i$ that satisfy
\[
  t( \tup x_1,\dots,\tup x_n, \tup a_{n+1} ) =
  t( \tup x_1,\dots,\tup x_n, \tup b_{n+1} )
\]
for all choices of $\tup x_i$'s between $\tup a_i$ and $\tup b_i$ except the case
where all are $\tup b_i$, we have
\[
  t( \tup b_1,\dots,\tup b_n, \tup a_{n+1} ) =
  t( \tup b_1,\dots,\tup b_n, \tup b_{n+1} ).
  \]
Recently, Moorhead \cite{Mo:HCTF} provided a~condition using two terms that is
equivalent to the `one term' condition given by Bulatov in the case that the
algebra lies in a~congruence modular variety. This gives another simple
description of supernilpotency: loosely speaking, for a~fixed $\tup a_i$ and
$\tup b_i$, the value of any term $s$ on the tuple $( \tup b_1,\dots,\tup
b_{n+1} )$ is determined by its values on all other tuples consisting of $\tup
a_i$ and $\tup b_i$ in the right order. The next theorem, which is based on the
results in \cite{Op:ARDO}, shows that this unique value can be obtained as
a~result of a~certain $(2^n-1)$-ary term (a~\emph{strong cube term}
\cite[p.\ 375]{Op:ARDO}) applied to the values of all the other tuples. We
recall that every algebra $\al A$ with a~Mal'cev term $q$ has a~strong $n$-cube
term $q_n$ for every $n>1$, and moreover such a~term can be obtained recursively
from the Mal'cev term by $q_2(x,y,z) = q(y,x,z)$ and
\[
  q_{n+1} (x_0,\dots,x_{2^{n+1}-1}) =
    q_2 ( q_n ( x_0,\dots,x_{2^n-2} ), x_{2^n -1},
      q_n ( x_{2^n},\dots,x_{2^{n+1}-2} )).
\]
Also, by \cite[Lemma~4.1]{Op:ARDO}, for every $n \ge 2$, every algebra with a
strong $n$-cube term has a Mal'cev term.  By a \emph{polynomial term} of the
algebra $\al A$, we understand a term of the algebra $\al A^*$, which is the
expansion of $\al A$ with one nullary constant operation for every element of
$A$. 

\begin{theorem} \label{thm:spn-identities}
  Let $n\geq 2$ and let $\al A$ be an~algebra with a~strong $n$-cube term $q_n$.
  Then the following are equivalent
  \begin{enumerate}
  \item $\al A$ is supernilpotent of degree $n-1$;
  \item \label{it:t2} for all $m_1, \ldots, m_n \in \N$, for all terms $t$ of
    arity $m = m_1 + \dots + m_n$, and for all $\tup a_1,\tup b_1 \in A^{m_1},
    \dots$, $\tup a_n, \tup b_n \in A^{m_n}$, we have
    \[
      q_n^{\al A} \bigl( t^{\al A} ( \tup a_1, \dots, \tup a_n ),
      t^{\al A} ( \tup b_1, \tup a_2, \dots, \tup a_n ),
      \dots,
      t^{\al A} ( \tup a_1, \tup b_2, \dots, \tup b_n ) \bigr) =
      t^{\al A} ( \tup b_1, \tup b_2, \dots, \tup b_n )
    ;\]
  \item for all $n$-ary polynomial terms $t$ and all $a_1, b_1,\dots, a_n, b_n
    \in A$, we have
    \[
      q_n^{\al A}\bigl( t^{\al A}(  a_1, \dots,  a_n ),
      t^{\al A}(  b_1,  a_2, \dots,  a_n ),
      \dots,
      t^{\al A}(  a_1,  b_2, \dots,  b_n ) \bigr) =
      t^{\al A}(  b_1,  b_2, \dots,  b_n ).
    \]
\end{enumerate}
\end{theorem}

The proof heavily relies on the properties of the relation $\Delta
(\alpha_1,\dots,\alpha_n)$.  In the case $\al A$ is a~Mal'cev
algebra, this relation of arity $2^n$ is described
in \cite[Lemma~3.3]{Op:ARDO} by
\begin{multline*}
 \Delta(\alpha_1,\dots,\alpha_n)   =
  \big\{ \big(
    t ( \tup a_1, \dots, \tup a_n ),
    t( \tup b_1, \tup a_2, \dots, \tup a_n ),
    \dots,
    t( \tup b_1, \tup b_2, \dots, \tup b_n )
    \big) \midbar \\
     \text{ for all } i \in \{1,2, \ldots, n\} : \, m_i \in \mathbb N_0, \tup
     a_i,\tup b_i \in A^{m_i}, (\tup a_i,\tup b_i) \in \alpha_i^{m_i}, \\
     \text { and } t\in \Clo_{\sum_{i = 1}^n m_i} \al A
  \big\}.
\end{multline*}

\begin{proof}
To simplify the notation, let $\Delta_n$ denote the relation $\Delta(1_{\al A},
\dots, 1_{\al A})$, where $1_{\al A}$ appears $n$ times, and let $[1]_n$
denote the $n$-ary commutator $[1_{\al A},\dots,1_{\al A}]$.

$(1)\Rightarrow(2)$: This implication is a~consequence of \cite[Lemma 4.2]{Op:ARDO}.
First, observe that
\(
  \big(
  t^{\al A} ( \tup a_1, \dots, \tup a_n ),
    t^{\al A} ( \tup b_1, \tup a_2 \dots, \tup a_n ),
    \dots,
    t^{\al A} ( \tup b_1, \tup b_2 \dots, \tup b_n )
    \big)  \in \Delta_n.
\)
From the mentioned lemma, we get that the last element of this tuple is
$[1]_n$-related to the result of $q_n$ applied to all the previous elements. But
since $\al A$ is supernilpotent of degree $n$, and therefore $[1]_n = 0_{\al
A}$, this gives the desired identity.

$(2)\Rightarrow(3)$: For the $n$-ary polynomial term $t(x_1,\ldots,x_n)$, there
is a term $s$ in the language of ${\al A}$ and there are $c_1,\ldots, c_m \in A$
such that \[ t(x_1,\ldots,x_n) = s(x_1,\ldots,x_n,c_1, \ldots,c_m). \]
Now we apply (2) for the term $s$ and for
${\tup a}_1 := a_1$, ${\tup b}_1 := b_1$, \ldots,
${\tup a}_{n-1} := a_{n-1}$, ${\tup b}_{n-1} := b_{n-1}$,
${\tup a}_n := (a_n, c_1,\ldots, c_m)$,
${\tup b}_n := (b_n, c_1,\ldots, c_m)$.

$(3)\Rightarrow(1)$:
We will prove that the condition (3) implies that for any $n$-tuple of principal
congruences $\theta_1,\dots,\theta_n$, we have $[\theta_1,\dots,\theta_n] =
0_{\al A}$.
The claim then follows from join distributivity of the higher commutator. Suppose
that $\theta_i = \Cg{(a_i,b_i)}$ for all $i$.  We know that the relation
$\Delta(\theta_1,\dots,\theta_n)$ \cite[Lemma 3.3]{Op:ARDO} consists of tuples
of the form
\[
  \big(
    t( \tup a_1, \dots, \tup a_n ),
    t( \tup b_1, \tup a_2, \dots, \tup a_n ),
    \dots,
    t( \tup b_1, \tup b_2, \dots, \tup b_n )
  \big)
\]
where $\tup a_i \equiv_{\theta_i} \tup b_i$ and $t$ is a term operation of ${\al
A}$.  Since $\al A$ is a~Mal'cev algebra, and therefore any reflexive binary
compatible relation is a~congruence, and since $\delta_i$ is generated by
$(a_i,b_i)$, every pair $(c,d)\in \theta_i$ is of the form $(p^{\al A}(a_i),
p^{\al A}(b_i))$, where $p$ is a polynomial term of ${\al A}$.  Hence, $\tup
a_i$ and $\tup b_i$ are of the form $(t_{i1}(a_i),\dots,t_{im_i}(a_i))$ and
$(t_{i1}(b_i),\dots,t_{im_i}(b_i))$ for some unary polynomial operations
$t_{ij}$. By composing these polynomials with $t$, we obtain
\begin{multline*}
  \Delta(\theta_1,\dots,\theta_n) = 
  \{ \big(
    p( a_1, \dots, a_n ),
    p( b_1, a_2, \dots, a_n ),
    \dots,
    p( b_1, b_2, \dots, b_n )
    \big)
    \mid \\
    p \text{ is an an~$n$-ary polynomial operation of } {\al A} \}.
\end{multline*}
By combining this observation with $(3)$, we obtain that the last coordinate of
a~tuple in $\Delta(\theta_1,\dots,\theta_n)$ is determined by the other
coordinates, therefore by \cite[Theorem 1.2]{Op:ARDO}, we get that
$[\theta_1,\dots,\theta_n] = 0_{\al A}$, as required.
\end{proof}

It is a consequence of \cite[Lemma~2.7]{BM:SPD} (which builds upon Lemma~14.6 of
\cite{FM:CTFC}) that a finite supernilpotent algebra in a congruence permutable
variety has a finitely generated clone of term operations.
Theorem~\ref{thm:spn-identities} provides another way of establishing this fact.

\begin{corollary}[{cf. \cite[Lemma~2.7]{BM:SPD}}]  \label{cor:fg}
  Let $n \in \N$, and let $\ab{A}$ be a supernilpotent Mal'cev algebra
  of degree $n$. Then the clone of term operations is generated by
  the Mal'cev term operation together with all term operations of arity at most
  $n+1$.
\end{corollary}

\begin{proof}
  By induction on $k$, we show that every $k$-ary term operation of $\ab{A}$
  can be generated. We use \cite[Lemma~4.1]{Op:ARDO} to produce a strong
  cube term $q_{n+1}$ of arity $2^{n+1} - 1$ for $\ab{A}$.
  Let $k \ge n+2$, and let
  $f(x_1, x_2,\ldots, x_k)$ be a $k$-ary term operation.
  We set
  $\tup b_1 = x_{1}$, \ldots, $\tup b_{n}  = x_{n}$, 
  $\tup b_{n + 1} = (x_{n+1}, x_{n+2}, \ldots, x_{k})$,
  $\tup a_1 = \dots = \tup a_n = x_{n+1}$, and
  $\tup a_{n+1}$ to the $(k - n)$-tuple  $(x_{n+1},\dots,x_{n+1})$.
  Theorem~\ref{thm:spn-identities}\,\eqref{it:t2} implies
  \begin{multline*}
    q_{n+1}^{\ab{A}} (f (x_{n+1}, x_{n+1}, \dots, x_{n+1}, \dots, x_{n+1}), 
                  f (x_1, x_{n+1}, \dots,    x_{n+1}, \dots, x_{n+1}), 
                \dots, \\
                f (x_{n+1}, x_2, \dots,    x_{n+1}, \dots, x_k))
                  =
                f (x_1, \dots, x_k).
  \end{multline*}               
  Each of the $2^{n+1} - 1$ arguments of $q_{n+1}^{\ab{A}}$ contains at least
  two occurrences of $x_{n+1}$ and is therefore of essential arity at most $k-1$.
  By the induction hypothesis, each of these arguments describes a function
  that lies in the
  clone generated by the Mal'cev operation and the $(n+1)$-ary functions.
  Since $q_{n+1}$ is composed from the Mal'cev term, 
  $f$ can be generated by the Mal'cev term and functions of arity
  at most $n+1$.
\end{proof}

This generalizes \cite[Proposition~6.18]{AM:SAOH} to clones that do not
contain all constant operations. In contrast to the constantive case,
term functions of arity $n$ may not suffice: as an example
consider the clone $C$ on the set $M_{2\times 2}(\mathbb Z_2)$ of $2\times 2$
matrices over $\mathbb Z_2$ that contains all functions
$(X_1,\ldots, X_k) \mapsto \sum_{i = 1}^k A_iX_i$ with
$A_1, \ldots, A_k \in M_{2\times 2}(\mathbb Z_2)$ and $\sum_{i=1}^k A_i = 1$.
Then the algebra $\ab{A} = (M_{2\times 2}(\mathbb Z_2), C)$ is
$1$-supernilpotent, but $C$ is not generated by the identity mapping and the
unique Mal'cev operation in the clone. 

The condition~\eqref{it:t2} in Theorem~\ref{thm:spn-identities} also provides an
explicit, though infinite, set of identities that defines supernilpotency in
a~Mal'cev variety:

\begin{corollary}
  Let $\var V$ be a~variety with a~strong cube term $q_n$. Then the class of all
  supernilpotent algebras of degree $n$ forms a~subvariety of $\var V$. This
  subvariety is defined by the collection of identities of the form
  \[
    q_n\bigl( t( \tup x_1, \dots, \tup x_n ),
    t( \tup y_1, \tup x_2, \dots, \tup x_n ),
    \dots,
    t( \tup x_1, \tup y_2, \dots, \tup y_n ) \bigr) \equals
    t( \tup y_1, \tup y_2, \dots, \tup y_n ),
  \]
  where $t$ is a term of arity $k \geq n$ and  $k_1,\dots,k_n \in \N$ are such
  that $k_1 + \dots + k_n = k$.  Here, $\tup x_i$ and $\tup y_i$ denote the
  tuples of variables $(x_{i1},\dots,x_{ik_i})$ and $(y_{i1},\dots,y_{ik_i})$,
  respectively.
\end{corollary}

We will now use these identities to express terms of higher arity using the
strong cube term $q_k$ and terms of smaller arity. This method allows us to
prove that there is a~logarithmic bound on the sequence $\He{\al A}n$ for every
supernilpotent finite algebra with a~Mal'cev term, and as a consequence, we
obtain a~polynomial bound on $\Le{\al A}n$.

\begin{theorem} \label{thm:spn-height}
  Let $\al A$ be a~finite supernilpotent Mal'cev algebra.  Then there exist
  positive real numbers $c_1, c_2$ such that for all $n \in \N$,
  \(
    \He{\al A}{n} \leq c_1 \log n + c_2.
  \)
\end{theorem}

\begin{proof}
Let $k$ be $1$ plus the degree of supernilpotency of $\al A$, and let $q_k$ be
a~strong cube term of ${\al A}$. Further assume that $\he{q_k} = d$.  We will
prove that there is a~constant $c<1$ such that for any large enough $n$, we have
\(
  \He{\al A}{n} \leq d + \He{\al A}{\lfloor cn\rfloor}.
\)
To do that, we start with a~term $f$ in the language of $\al A$ of high-enough
arity $n$, and we will group its variables into $k$ pieces of almost the same
length, and then use the identity from item~\eqref{it:t2} in
Theorem~\ref{thm:spn-identities} to replace $f$ by a composition of the strong
cube term $q_k$ with terms of arity lower than $n$.
More precisely, let $n = qk + r$ where $r < k$, $q > 1$. We group the variables
of $f$ into $r$ many $(q+1)$-tuples $\tup x_1$, \dots, $\tup x_r$ and $(k - r)$
many $q$-tuples $\tup x_{r+1}$, \dots, $\tup x_k$ so that $\tup x_1 =
(x_1,\dots,x_{i_1})$, $\tup x_2 = (x_{i_1+1},\dots,x_{i_2})$, etc.  We take
a~new variable $y$, and for $i \in \{1,\dots,k\}$, we let $\tup y_i$ denote the
tuple $(y,\dots,y)$ of the same length as $\tup x_i$ (i.e., $\tup y_i$ is a
$(q+1)$-tuple for $i \leq r$ and a $q$-tuple for $i > r$).  Now applying the
condition~\eqref{it:t2} of Theorem~\ref{thm:spn-identities}, we get that
\begin{equation} \label{eq:qf}
  f(\tup x_1,\dots,\tup x_k) \equals
  q_k( f(\tup y_1,\tup x_2,\dots,\tup x_k), \dots,
      f(\tup x_1,\tup y_2, \dots, \tup y_k))
\end{equation}
is satisfied in $\al A$. The right hand side is an application of $q_k$ on terms
obtained from $f$ by substituting one or more of $\tup x_i$'s by $\tup y_i$. The
maximal arity of these $2^k - 1$ terms is obtained, e.g., when only $\tup x_k$
is substituted by $\tup y_k$.
In this case, omitting $\tup x_k$ reduces the arity of $f$ by $q = \lfloor
\frac{n}{k} \rfloor$ and adds $1$ for the new variable $y$.
Hence, each of the $2^k - 1$ arguments of $q_k$ in~\eqref{eq:qf} contains at most
$n + 1 - \lfloor \frac nk \rfloor$ many different variables.
For each of these $2^k - 1$ arguments, we pick a~term $u_i$ of
height at most $\He{\al A}{n + 1 - \lfloor \frac nk \rfloor}$ representing
the same function on ${\al A}$. From Lemma~\ref{lem:heightfromthemiddle}, we
obtain that $q_k (u_1, \ldots, u_{2^{k} - 1})$ is a~term of height at most
$d + \Ht_{\al A} ( n + 1 - \lfloor \frac nk \rfloor )$
that induces the same function on ${\al A}$ as $f$. Therefore, 
\[
  \Ht_{\al A} (n) \leq d + \Ht_{\al A} ( n + 1 - \lfloor \frac nk \rfloor )
\]
for every $n \in \mathbb N$. We choose $\epsilon \in \mathbb{R}$ such that $0 <
\epsilon < 1/k$, we set $c = 1 - 1/k + \epsilon$, and let $n_0 > k$ be big
enough so that $\epsilon n_0 \geq 2$ and $c n_0 \ge 1$.  Then for any $n \geq
n_0$, we have
\[
  n + 1 - \lfloor \frac nk \rfloor < n  - \frac nk + 2
    \leq (1 - \frac 1k + \epsilon)n = c n.
\]
Therefore, $\Ht_{\al A} (n) \leq d + \Ht_{\al A} ( \lfloor cn \rfloor )$ for any
$n \geq n_0$. From Lemma~\ref{lem:recursion}, we obtain that $\Ht_{\al A} (n)
\le d\log_{1/c} n + \Ht_{\al A} (n_0)$ for all $n \in \N$.  Choosing $c_1 : = d
/ \log (1/c)$ and $c_2 = \He{\al A}{n_0}$ we obtain the required result.
\end{proof}

\begin{corollary} \label{cor:lsm}
  Let $\al A$ be a~finite supernilpotent Mal'cev algebra, then
  there exist an~integer $k>0$ and a positive real $c$ such that
  for all $n \in \N$,
  \(
    \Le{\al A}{n} \leq c n^k.
  \)
\end{corollary}
\begin{proof}
  The algebra $\al A$ need not be of finite type. However, by
  Corollary~\ref{cor:fg}, its clone of term operations is finitely generated.
  Therefore, we can choose a finite subset of the fundamental operations of $\A$
  that generates all other fundamental operations, and we let $\A'$ be the
  reduct of $\A$ with only these finitely many fundamental operations; let $m$
  be their maximal arity.  By Theorem~\ref{thm:spn-height}, we have $\He{\al
  A'}{n} \le c_1 \log (n) + c_2$. Now Lemma~\ref{lem:general-bound} yields
  $\Le{\al A'}{n} \le (m+1)^{\He{\al A'}{n}}$, and thus there is a positive real
  $c$ and $k \in \N$ such that for all $n \in \N$, $\Le{\al A'}{n} \le c n^k$.
  We clearly have $\Le{\al A}{n} \le \Le{\al A'}{n}$ for all $n \in \N$ which
  implies the result.
\end{proof}

\begin{proof}[Proof of Theorem~\ref{thm:supernilpotent}.]
  Since the variety generated by $\al A$ is locally finite and has a~weak
  difference term, using \cite[Theorem~4.8]{Wi:OSA} and its proof, we get that
  $\ab{A}$ has a Mal'cev term.  Corollary~\ref{cor:lsm} now yields the second
  inequality.  For the first equality, we show that for every $k \in \N_0$,
  there is an $(2k+1)$-ary term $t_k$ such that $t^{\ab{A}}$ depends on all of
  its arguments. To this end, let $m$ be a Mal'cev term, let $t_0 (x_1) = x_1$,
  and $t_{k} (x_1,\ldots, x_{2k+1}) = m(t_{k-1} (x_1,\ldots, x_{2k-1}), x_{2k},
  x_{2k + 1})$ for $k \in \N$.  Let $a,b$ be different elements of $A$.  We
  consider $y := t_k^{\ab{A}}(a,\ldots, a, b, \ldots, b)$ where the first $2 k +
  1 - l$ arguments are $a$ and the remaining $l$ arguments are set to $b$. Then
  $y = a$ if $l$ is even, and $y =b$ otherwise. This proves that $t_k$ depends
  on  all of its $2k+1$ arguments.  We will now show $\Le{\al A}{n} \ge n - 1$:
  if $n$ is even, then $t_{(n-2)/2}^{\ab{A}} (x_1,\ldots, x_{n-1})$ depends on
  $n-1$ arguments. Hence every term representing $t_{(n-2)/2}^{\ab{A}}$ must
  contain at least $n-1$ variables, and is thus of length at least $n-1$.  If
  $n$ is odd, then $t_{(n-1)/2}^{\ab{A}} (x_1,\ldots,x_n)$ depends on all of its
  $n$ arguments, and thus $\lenop_{\ab{A}} (t_{(n-1)/2}^{\ab{A}}) \ge n$.
\end{proof}

In the rest of this section, we give an argument that out of finite algebras in
congruence modular varieties, only supernilpotent ones have a~polynomial bound
on the length of term functions. This argument is based on a~description of the
sequence $\Fs{\al A}n$. A~rough asymptotic behavior of this sequence for
congruence modular algebras have been first described by Kearnes in
\cite{Ke:CMVW}. He proved that an~algebra of finite type in a~congruence modular
variety has a~doubly exponential lower bound if and only if it is not a~product
of prime-power order nilpotent algebras which is now known to be equivalent to
being supernilpotent \cite{AM:SAOH}. We present a~refinement of this result
which is given by a~combination of several different sources.

\begin{proposition} \label{pro:comb}
   Let $\ab{A}$ be a finite algebra in a congruence modular variety, and let $k
   \in \N$.  Then $\ab{A}$ is $k$-supernilpotent if and only if there is a
   polynomial $p$ of degree $k$ such that for all $n \in \N$,
   $\Fs{\al A}n \le 2^{p(n)}$.
\end{proposition}

\begin{proof}
  For the ``if''-part, first observe that Theorem~9.18 of \cite{HM:TSOF} implies
  that the variety ${\mathcal V}(\ab{A})$ omits types $\boldsymbol{1}$ and
  $\boldsymbol{5}$.  From \cite[Lemma~12.4]{HM:TSOF}, we obtain that $\ab{A}$ is
  right nilpotent, and since the commutator operation in a congruence modular
  variety is commutative, $\ab{A}$ is therefore nilpotent. Now
  \cite[Theorem~6.2]{FM:CTFC} yields that $\ab{A}$ has a Mal'cev term.  Let
  $\ab{A}^*$ be the expansion of $\ab{A}$ with all its constants.  Then
  $\ab{A}^*$ is nilpotent and generates a congruence permutable variety.  The
  variety ${\mathcal V}(\ab{A}^*)$ is nilpotent by \cite[Theorem~14.2]{FM:CTFC},
  and hence congruence uniform by \cite[Corollary~7.5]{FM:CTFC}.  Since for all
  $n \in \N$,  $\F{\ab{A}^*}{n} \le \F{\ab{A}}{n + |A|} \le 2^{p (n + |A|)}$, we
  obtain from the proof of \cite[Theorem~1]{BB:FSON} that all commutator terms
  (in the sense of \cite[p.  179]{Ke:CMVW}) of $\ab{A}^*$ are of rank at most
  $k$. Hence all commutator polynomials (in the sense of
  \cite[Definition~7.2]{AM:SAOH}) of $\ab{A}$ are of rank at most $k$, and then
  \cite[Lemma~7.5]{AM:SAOH} yields that $\ab{A}$ is $k$-supernilpotent.

  For the ``only if''-part, we assume that $\ab{A}$ is $k$-supernilpotent. Then
  from the proof of \cite[Theorem 4.8]{Wi:OSA}, it follows that $\ab{A}$ has a
  Mal'cev term, and thus by Lemma~7.5 of \cite{AM:SAOH} each commutator term of
  $\ab{A}$ is of rank at most $k$. Now from the proof of Theorem~1 in
  \cite{BB:FSON}, we obtain a polynomial $p$ of degree at most $k$ such that for
  all $n \in \N$, $\F{\ab{A}}{n}$ has exactly $2^{p(n)}$ elements.
\end{proof}

Section~4 of \cite{Ai:OTDD} contains a self-contained version of
Proposition~\ref{pro:comb} for the case that $\ab{A}$ is an expanded group.

\begin{corollary}
  Let $\al A$ be a~finite algebra of finite type in a~congruence modular
  variety. Then the following are equivalent.
  \begin{enumerate}
    \item $\al A$ is supernilpotent;
    \item there exists constants $c_1$, $c_2 > 0$ such that $\He{\al A}n \leq c_1
    \log n + c_2$ for all $n > 0$;
    \item there exists a~polynomial $p_1$ such that $\Le{\al A}n \leq p_1(n)$
    for all $n>0$;
    \item there exists a~polynomial $p_2$ such that $\Fs{\al A}n \leq
    2^{p_2(n)}$ for all $n>0$.
  \end{enumerate}
\end{corollary}

\begin{proof}
  $(1)\Rightarrow(2)$: as noted in the proof of Proposition~\ref{pro:comb}, a~finite
  supernilpotent algebra has a~Mal'cev term, therefore
  Theorem~\ref{thm:spn-height} applies in this case.

  The implications $(2)\Rightarrow(3)$ and $(3)\Rightarrow(4)$ are given by
  Lemma~\ref{lem:general-bound}, and $(4)\Rightarrow(1)$ is implied by
  Proposition~\ref{pro:comb}.
\end{proof}

\subsection*{Acknowledgements}
The authors would like to thank \'{A}gnes Szendrei for supplying the argument
leading to Lemma~\ref{lem:general-bound}\,\eqref{it:2}.

\def\cprime{$'$}

\end{document}